\documentclass[11pt]{amsart}

\usepackage{amssymb}

\usepackage[linktocpage=true,hypertex]{hyperref}

\numberwithin{equation}{section}

\newtheorem{theorem}{Theorem}
\numberwithin{theorem}{section}
\newtheorem{proposition}[theorem]{Proposition}
\newtheorem{lemma}[theorem]{Lemma}
\newtheorem{corollary}[theorem]{Corollary}

\theoremstyle{definition}
\newtheorem{definition}[theorem]{Definition}
\newtheorem{remark}[theorem]{Remark}
\newtheorem{example}[theorem]{Example}

\newcommand{\GL}{{\operatorname{GL}}}

\newcommand{\Gmn}{G_{m,n}}
\newcommand{\id}{\operatorname{id}}
\newcommand{\Int}{\operatorname{Int}}
\newcommand{\inv}{^{-1}}
\newcommand{\lra}{\longrightarrow}
\newcommand{\Mat}{{\operatorname{M}}}
\newcommand{\M}{\operatorname{M}}
\newcommand{\Mn}{\Mat_n}
\newcommand{\Mnm}{(\Mn)^m}
\newcommand{\PGL}{{\operatorname{PGL}}}
\newcommand{\PGLn}{{\operatorname{PGL}_n}}
\newcommand{\pglntau}{\PGLn\rtimes\langle\tau\rangle}
\newcommand{\pntau}{P_{n,\tau}}
\newcommand{\PGO}{{\operatorname{PGO}}}
\newcommand{\PGOn}{\PGO_n}
\newcommand{\PGSpn}{{\operatorname{PGSp}_n}}
\newcommand{\Q}{\operatorname{Q}}

\newcommand{\Skew}{\operatorname{Skew}}
\newcommand{\Sym}{\operatorname{Sym}}
\newcommand{\tp}{^t}
\newcommand{\tr}{\operatorname{tr}}
\newcommand{\UD}{\operatorname{{\it UD}}}
\newcommand{\UDmn}{\UD(m,n)}
\newcommand{\Z}{\operatorname{Z}}

\begin{document}

\date{August 12, 2008} 

\title%
[Central simple algebras with involution]%
{Central Simple Algebras with Involution:\\ A Geometric Approach}

\author{Nikolaus Vonessen}
\address{Department of Mathematical Sciences, University of Montana,
  Missoula, MT 59812-0864, USA}
\email{\href{mailto://nikolaus.vonessen@umontana.edu}{nikolaus.vonessen@umontana.edu}}
\urladdr{\href{http://www.math.umt.edu/vonessen}{www.math.umt.edu/vonessen}}

\begin{abstract}
  Let $k$ be an algebraically closed base field of characteristic
  zero.  The category equivalence between central simple algebras and
  irreducible, generically free $\PGLn$-varieties is extended to the
  context of central simple algebras with involution.  The associated
  variety of a central simple algebra with involution comes with an
  action of \hbox{$\pglntau$}, where $\tau$ is the automorphism of
  $\PGLn$ given by $\tau(h)=(h\inv)^{\text{transpose}}$.  Basic properties of an
  involution are described in terms of the action of \hbox{$\pglntau$}
  on the associated variety, and in particular in terms of the
  stabilizer in general position for this action.
\end{abstract}

\subjclass[2000]{16W10, 14L30, 16K20}


\keywords{Central simple algebra with involution, associated variety,
  group action, stabilizer in general position, rational quotient}

\maketitle

\tableofcontents

\section{Introduction}

We work over an algebraically closed base field $k$ of characteristic
zero.  In this paper, a central simple algebra is a simple algebra
finite-dimensional over its center with the additional property that
the center is a finitely-generated field extension of $k$.  We begin
by reviewing some facts, mostly from~\cite{rv1}, about central simple
algebras and their associated varieties.  Given an irreducible,
generically free $\PGLn$-variety $X$, $k_n(X)$ denotes the set of
$\PGLn$-equivariant rational maps from $X$ to $\Mn$, the algebra of
$n\times n$ matrices over $k$.  Under the usual addition and
multiplication of maps into $\Mn$, $k_n(X)$ turns out to be a central
simple algebra; its center can be identified with the fixed field
$k(X)^\PGLn$ (\cite[Lemma~8.5]{reichstein}; see also
\cite[Lemma~2.8]{rv2}).

Conversely, let $A$ be a central simple algebra (whose center is, by
our general convention, a finitely-generated field extension of $k$).
Then there is an irreducible, generically free $\PGLn$-variety $X$
such that $A$ is isomorphic to $k_n(X)$.  In fact, $X$ is unique up to
birational isomorphism of $\PGLn$-varieties, and it is called the
$\PGLn$-variety {\em associated} to $A$ (\cite[Theorem~7.8 and
Section~8]{rv1}).

Now let $A$ and $B$ be central simple algebras with associated
$\PGLn$-varieties $X$ and $Y$, respectively.  A $\PGLn$-equivariant
dominant rational map $\varphi\colon X\dasharrow Y$ induces an algebra
homomorphism (necessarily an embedding) $\varphi^*\colon
k_n(Y)\hookrightarrow k_n(X)$ defined by $f\mapsto f\circ\varphi$.
Conversely, an algebra homomorphism $\alpha\colon
k_n(Y)\hookrightarrow k_n(X)$ induces a $\PGLn$-equivariant dominant
rational map $\alpha_*\colon X\dasharrow Y$.  The expected functorial
properties hold: $(\varphi^*)_*=\varphi$, $(\alpha_*)^*=\alpha$,
$(\id_{k_n(X)})_*=\id_X$, $(\id_X)^*=\id_{k_n(X)}$,
$(\psi\circ\varphi)^*=\varphi^*\circ\psi^*$, and
$(\beta\circ\alpha)_*=\alpha_*\circ\beta_*$ (\cite[Section~8]{rv1}).
In fact, the functor defined by
\[X\mapsto k_n(X),\quad \bigl(\varphi\colon X\dasharrow Y\bigr)
   \mapsto \bigl(\varphi^*\colon k_n(Y)\hookrightarrow k_n(X)\bigr)\]
is a contravariant equivalence of categories from the category
$\textit{Bir}_n$ of irreducible, generically free $\PGLn$-varieties
(with dominant rational $\PGLn$-equivariant maps as morphisms) to the
category $\textit{CS}_n$ of central simple algebras of degree $n$
(with algebra homomorphisms as morphisms) (\cite[Theorem~1.2]{rv1}).

\subsection*{The Results.}
We denote by $\tau$ the automorphism of order two of $\PGLn$ defined
by 
\[\tau(h)=(h\inv)\tp=(h\tp)\inv\,,\]
where $t$ denotes transposition.  Central for our work is the
semidirect product
\[\pntau:=\pglntau\,.\]
Throughout, we identify $\PGLn$ with the subgroup $\PGLn\times 1$ of
$\pntau$.  

Given a $\pntau$-variety $X$ which is $\PGLn$-generically free, the
central simple algebra $k_n(X)$ has an involution $\sigma_{X,\tau}$
induced by the action of $\tau$ on $X$:  For $f\in k_n(X)$,
\[\sigma_{X,\tau}(f)=\text{transposition}\circ f\circ \tau\,,\]
see Lemma~\ref{lem4}.  Conversely, given a central simple algebra $A$
of degree $n$ with involution $\sigma$, there is an irreducible
$\pntau$-variety $X$ which is $\PGLn$-generically free such that
$(A,\sigma)$ and $(k_n(X),\sigma_{X,\tau})$ are isomorphic as algebras
with involution (Proposition~\ref{prop6}).  Moreover, $X$ is unique up
to birational isomorphism of $\pntau$-varieties
(Corollary~\ref{cor-to-lem9}).  We call $X$ the $\pntau$-variety {\em
  associated} to $(A,\sigma)$.  A model for $X$ can usually be found
in $\Mnm$ for some $m$ (unless $n=2$ and the involution is symplectic),
see Lemma~\ref{lem:new}.

The category equivalence between $\textit{Bir}_n$ and $\textit{CS}_n$
extends to this setting.  In the sequel, the morphisms in categories
of algebras with involution are the $k$-algebra homomorphisms
preserving involutions.  And on the geometric side, the morphisms are
the dominant rational maps which are equivariant for the action of the
given group.

\begin{theorem}\label{thm}
  There is a contravariant category equivalence
\[
\begin{array}{ccc}
\fbox{\parbox{4.4cm}{\begin{center}
Irreducible $\pntau$-varieties which are generically free as $\PGLn$-varieties
\end{center}}}
&\!\!\longleftrightarrow\!\!&
\fbox{\parbox{4.4cm}{\begin{center}
Central simple algebras of degree~$n$ with involution
\end{center}}}\hphantom{\,.}
\end{array}
\]
given by the contravariant functor
  \[ \begin{array}{ccc}
  \parbox{4.4cm}{\begin{center} $X$ \end{center}}
    &\!\!\longmapsto\!\! &
  \parbox{4.4cm}{\begin{center}$(k_n(X),\sigma_{X,\tau})$\end{center}}\,\, \\
  \bigl(\varphi \colon X \dasharrow Y\bigr)
    &\!\!\longmapsto\!\!
         &      \bigl(\varphi^* \colon k_n(Y) \hookrightarrow k_n(X)\bigr) \,.
     \end{array}
  \]
\end{theorem}

This result is proved at the end of Section~\ref{sec:cat-equivalence}.
We show next how basic properties of an involution of a central simple
algebra are related to properties of the $\pntau$-action on the
associated variety.  Recall that an action of a linear algebraic
group~$G$ on a variety~$X$ has a {\em stabilizer in general position}
if there is a subgroup $S$ of $G$ such that for $x\in X$ in general
position, the stabilizer of $x$ in $G$ is conjugate to $S$.

\begin{theorem}\label{thm-stab}
  Let $A$ be a central simple algebra of degree~$n$ with
  involution~$\sigma$ and associated $\pntau$-variety~$X$.  Then the
  $\pntau$-action on $X$ has a stabilizer in general position, say
  $S$.
  \begin{itemize}
  \item[\upshape{(a)}]The involution $\sigma$ is of the first kind iff
    $S\neq1$.
  \item[\upshape{(b)}]The involution $\sigma$ is of the second kind
    {\upshape(}i.e, is unitary{\upshape)}
    iff $S=1$ {\upshape(}i.e., the action of $\pntau$ on $X$ is
    generically free{\upshape)}. 
  \item[\upshape{(c)}]The involution $\sigma$ is orthogonal iff $S$ is
    conjugate to $\{1,\tau\}$.
  \item[\upshape{(d)}]The involution $\sigma$ is symplectic iff $S$ is
    conjugate to $\bigl\{1,\tau 
    \bigl[\begin{smallmatrix}0&\,\,I\\-I&\,\,0\end{smallmatrix}\bigr]\bigr\}$.
  \end{itemize}
\end{theorem}

In part~(d), $g_0=
\bigl[\begin{smallmatrix}0&\,\,I\\-I&\,\,0\end{smallmatrix}\bigr]$
denotes the image in $\PGLn$ of the matrix
$\bigl(\begin{smallmatrix}0&\,\,I\\-I&\,\,0\end{smallmatrix}\bigr)$,
where $I$ is the identity matrix of order~$\frac{n}{2}$.  Note that
since $S$ is only defined up to conjugacy, one can always take it to
be $\{1,\tau\}$ if $\sigma$ is orthogonal, and to be $\{1,\tau g_0\}$
if $\sigma$ is symplectic.

This theorem is proved in Section~\ref{sec:orth-and-symp}.  The
existence of the stabilizer in general position follows easily from a
deep theorem of Richardson (see Remark~\ref{rem:richardson}) but will
be deduced independently from the two well-known facts that all
symmetric matrices in $\Mn$ are congruent to each other, and that all
skew-symmetric matrices in $\Mn$ are congruent to each other.  These
facts are also otherwise needed in the proof of
Theorem~\ref{thm-stab}, and this approach makes its proof only
marginally longer.

We conclude with two remarks.

\begin{remark}\label{rem:1.3}
  Let $\sigma$ be an involution of the first kind.  Whether $\sigma$
  is orthogonal or symplectic can also be described in the following,
  more geometric fashion: The proof of Theorem~\ref{thm-stab} in
  Section~\ref{sec:orth-and-symp} shows that in the case of
  involutions of the first kind, for all $x\in X$ in general position,
  there is a unique $g_x\in \PGLn$ such that $\tau(x)=g_x(x)$;
  moreover, there are two mutually exclusive cases: either the
  preimages of $g_x$ in $\GL_n$ are symmetric matrices for all $x\in
  X$ in general position, or they are skew-symmetric matrices for all
  $x\in X$ in general position (cf.~Lemma~\ref{lem:U}).  Accordingly,
  we say that the action of $\tau$ on $X$ is of {\em symmetric} or
  {\em skew-symmetric} type.  Finally, $\sigma$ is orthogonal if and
  only if the action of $\tau$ on $X$ is of symmetric type, and
  $\sigma$ is symplectic if and only if the action of $\tau$ on $X$ is
  of skew-symmetric type (see Lemma~\ref{lem:U}).  See
  Section~\ref{sec:ex-rho-3} for an example.
\end{remark}

\begin{remark}\label{rem:cat-equivalences}
  Theorems~\ref{thm} and~\ref{thm-stab} can be used to give a new
  proof of the following well-known contravariant category
  equivalences (see\cite[\S29]{KMRT} for similar results under less
  restrictive hypotheses):
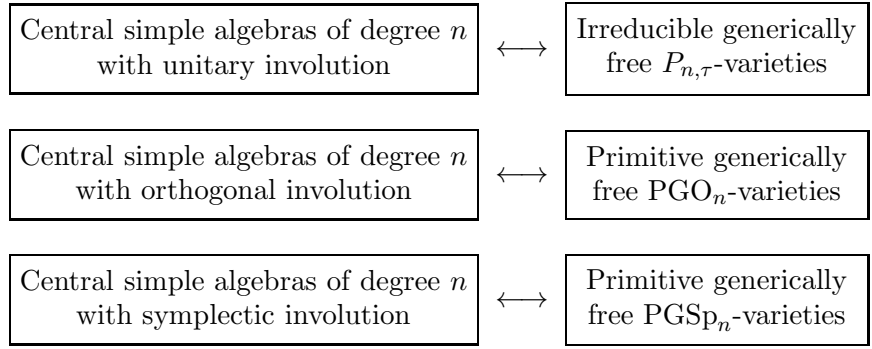
\[
\begin{array}{ccc}
\fbox{\parbox{6cm}{\begin{center}
Central simple algebras of degree~$n$ with unitary involution
\end{center}}}
&\!\!\longleftrightarrow\!\!&
\fbox{\parbox{3.8cm}{\begin{center}
Irreducible generically free
$\pntau$-varieties 
\end{center}}}
\\ \\
\fbox{\parbox{6cm}{\begin{center}
Central simple algebras of degree~$n$ with orthogonal involution
\end{center}}}
&\!\!\longleftrightarrow\!\!&
\fbox{\parbox{3.8cm}{\begin{center}
Primitive generically free
$\PGOn$-varieties 
\end{center}}}
\\ \\
\fbox{\parbox{6cm}{\begin{center}
Central simple algebras of degree~$n$ with symplectic involution
\end{center}}}
&\!\!\longleftrightarrow\!\!&
\fbox{\parbox{3.8cm}{\begin{center}
Primitive generically free
$\PGSpn$-varieties 
\end{center}}}
\end{array}
\]
The first equivalence follows immediately from Theorems~\ref{thm}
and~\ref{thm-stab}.  For the others, one needs in addition the
following fact: {\em Let $G$ be a linear algebraic group, $S$ a closed
  subgroup and $N$ the normalizer of $S$ in $G$.  There is a category
  equivalence between the category of generically free primitive
  $N/S$-varieties and the category of primitive $G$-varieties with
  stabilizer~$S$ in general position}, %
cf.~\cite{kordonskii}, or \cite[Lemma~3.2]{rv-affine-stable} (which
in turn is based on \cite[Section~1.7]{popov}).  Note that if
$G=\pntau$ and $S$ is equal to $\langle\tau\rangle$ or $\langle\tau
g_0\rangle$, then $N/S$ is isomorphic to $\PGOn$ or $\PGSpn$,
respectively, as easy computations show.
\end{remark}

\subsection*{Conventions.}
As already stated, $k$ is an algebraically closed base field of
characteristic zero, and a central simple algebra is a simple algebra
finite-dimensional over its center with the additional property that
the center is a finitely-generated field extension of $k$.  All
algebra homomorphisms and involutions are assumed to be $k$-linear.
We always denote by $m$ and $n$ integers $\geq 2$.  We write $\Mn$ for
$\Mn(k)$, $\GL_n$ for $\GL_n(k)$, $\PGLn$ for $\PGLn(k)$, etc.  The
transpose of a matrix $a$ is denoted by $a\tp$.  If
$a=(a_1,\ldots,a_m)$ is an $m$-tuple of matrices, then
$a\tp=(a_1\tp,\ldots,a_m\tp)$.

\subsection*{Acknowledgments.}
The author is grateful to Zinovy Reichstein for helpful comments and
suggestions.

\section{A Category Equivalence}
\label{sec:cat-equivalence}
 
The generic matrix ring 
\[\Gmn=k\{X_1,\ldots,X_m\}\]
has an involution $\rho$ which is the $k$-linear map which reverses
the order of any monomial in the $m$ generic $n\times n$-matrices
$X_1,\ldots,X_m$.  That is,
\begin{equation}\label{eq-rho}
  \rho(X_{i_1}X_{i_2}\cdots X_{i_j})
        = X_{i_j}\cdots X_{i_2} X_{i_1}\,. 
\end{equation}
To see that $\rho$ is well-defined, consider the correspondingly
defined map $\rho_F$ on the free algebra $F=k\{x_1,\ldots,x_m\}$.  It
suffices to show that if $p=p(x_1,\ldots,x_m)\in F$ is a polynomial
identity for $n\times n$-matrices, then so is $\rho_F(p)$.  Well, let
$a=(a_1,\ldots,a_m)$ be any $m$-tuple of $n\times n$-matrices.  Then
\[\rho_F(p)(a_1\tp,\ldots,a_m\tp)=p(a_1,\ldots,a_m)\tp\,.\]
That is,
\[\rho_F(p)(a_1,\ldots,a_m)=p(a_1\tp,\ldots,a_m\tp)\tp=0\,.\]

The involution $\rho$ of $\Gmn$ is induced from transposition of
$m$-tuples of $n\times n$ matrices in the sense that for any 
$p=p(X_1,\ldots,X_m)\in\Gmn$,
\[\rho(p)(a_1,\ldots,a_m)=p(a_1\tp,\ldots,a_m\tp)\tp\,,\]
or
\begin{equation}\label{eq1}
  \rho(p)(a)=p(a\tp)\tp\,.
\end{equation}

\begin{lemma}\label{lem4}
  Let $X$ be an irreducible $\pntau$-variety which is
  $\PGLn$-generically free, so that $A=k_n(X)$ is a central simple
  algebra of degree~$n$.  Then the action of $\tau$ on $X$ induces an
  involution $\sigma_{X,\tau}$ of $A$ as follows: For $f\in A$,
  \[\sigma_{X,\tau}(f)=\text{\upshape transposition}\circ f\circ\tau\,.\]
  That is, for $x\in X$ in general position,
  \begin{equation}\label{eq2}
    \sigma_{X,\tau}(f)(x)=f(\tau x)\tp\,.
  \end{equation}
\end{lemma}

Note the similarity of formulas~\eqref{eq1} and~\eqref{eq2}.

\begin{proof}
For simplicity, set $\sigma=\sigma_{X,\tau}$.
It is clear that $\sigma(f)$ is a rational map $X\dasharrow\Mn$.
To check that it is $\PGLn$-equivariant, let $h\in\PGLn$ and $f\in A$.
Keeping in mind that $\PGLn$ acts on $\Mn$ by conjugation, and that $f$
is $\PGLn$-equivariant, one sees that for $x\in X$ in general
position,
\begin{align*}
\sigma(f)(hx) &= f(\tau h x)\tp = [f((h\inv)\tp\tau x)]\tp
               = [(h\inv)\tp f(\tau x) h\tp]\tp \\
              &= h f(\tau x)\tp h\inv
               = h [\sigma(f)(x)] h\inv\,.
\end{align*}
That is, $\sigma(f)$ is $\PGLn$-equivariant and belongs thus to $A$.

It is clear that $\sigma^2$ is the identity map of $A$, and that
$\sigma$ is $k$-linear.  Finally, let $f,g\in A$.  Then for $x\in X$
in general position,
\[\sigma(fg)(x) = [(fg)(\tau x)]\tp = g(\tau x)\tp \cdot f(\tau x)\tp
 =[\sigma(g)\sigma(f)](x)\,.\]
So $\sigma(fg)=\sigma(g)\sigma(f)$, and $\sigma$ is indeed an
involution of $A$.
\end{proof}

\begin{example} {\it The involution $\rho$ of $\UDmn$.}
  \label{example:rho} %
  The universal division algebra $\UDmn$ is obtained from $\Gmn$ by
  localizing at the non-zero central elements.  It can also be
  obtained from $\Gmn$ by inverting only the nonzero central symmetric
  elements (since for $p,s \in \Gmn$ with $s$ nonzero and central,
  $ps\inv=(p\rho(s)) (s\rho(s))\inv)$.  Hence the involution $\rho$ of
  $\Gmn$ extends to an involution of $\UDmn$, which we will also
  denote by $\rho$ (cf.\ \cite[Proposition~2.13.17]{rowenI}).  One
  checks easily that for all $p,s\in \Gmn$ with $s$ nonzero and
  central, $\rho(ps\inv)=\rho(p)\rho(s)\inv$.

  The associated variety of $\UDmn$ is $X=\Mnm$, where $\PGLn$ acts by
  simultaneous conjugation.  Letting $\tau$ act by simultaneous
  transposition turns $X=\Mnm$ into a
  $\pntau$-variety which is
  $\PGLn$-generically free.  Formula~\eqref{eq1} says that for all
  $p\in\Gmn$ and all $x\in X$, $\rho(p)(x)=p(\tau
  x)\tp=\sigma_{X,\tau}(x)$.  A straightforward calculation shows that
  for all $f\in\UDmn$, $\rho(f)(x)=\sigma_{X,\tau}(x)$ for all $x\in
  X$ in general position.  Hence $\rho=\sigma_{X,\tau}$.  We will
  return to this example in Example~\ref{example:rho:first-kind}
  and in Section~\ref{sec:ex-rho-3}.
  \hfill\qedsymbol
\end{example}

\begin{proposition}\label{prop6}
  Let $A$ be a central simple algebra of degree~$n$ with
  involution~$\sigma$.  As always, we assume that the center of $A$ is
  a finitely generated field extension of $k$.  Then there is an
  irreducible $\pntau$-variety $X$ which is $\PGLn$-generically free
  such that $(A,\sigma)$ and $(k_n(X),\sigma_{X,\tau})$ are isomorphic
  as algebras with involution.
\end{proposition}

Here $\sigma_{X,\tau}$ is the involution defined in Lemma~\ref{lem4}.
We call $X$ the {\it $\pntau$-variety associated to} ($A$, $\sigma$).
We will see below that $X$ is unique up to birational isomorphism of
$\pntau$-varieties (Corollary~\ref{cor-to-lem9}).  Lemma~\ref{lem:new}
shows that up to birational isomorphism, a model for $X$ can in most
cases be found in $\Mnm$.

\begin{proof}
  There is a finite-dimensional, $\sigma$-stable $k$-subspace $V$ of
  $A$ such that the $k$-algebra $R$ generated by $V$ is prime, and
  such that $A$ is the total ring of fractions of $R$.  Choose a basis
  $r_1,\ldots,r_m$ of $V$ consisting of eigenvectors of $\sigma$.
  More precisely, $\sigma(r_i)=\epsilon_ir_i$ where $\epsilon_i=\pm1$.
  Note that $R=k\{r_1,\ldots,r_m\}$.
  
  Consider the homomorphism $\phi\colon\Gmn\lra R$ defined by
  $X_i\mapsto r_i$.  Denote its kernel by $I$.  Recall that
  $\mathcal{Z}(I)$ is the set of common zeroes of the elements of $I$
  among the points in a certain $\PGLn$-stable dense open subset of
  $\Mnm$, see \cite[Definition~3.1]{rv1}.  Set $X=\mathcal{Z}(I)$.
  Then by construction, $X$ is an associated variety of $A$, see
  \cite[Theorems~6.4 and~7.8]{rv1}.  Moreover, $\phi$ induces
  $k$-algebra isomorphisms from $\Gmn/I$ onto $R$, and from
  $\Q(\Gmn/I)=k_n(X)$ onto $A$.  We will denote both of these
  isomorphisms by $\bar\phi$.
  
  We first lift $\sigma$ to an involution $\tilde\sigma$ of $\Gmn$.
  For $p(X_1,\ldots,X_m)\in\Gmn$, define
  \[\tilde\sigma(p)=\rho(p(\epsilon_1X_1,\ldots,e_mX_m))\,,\]
  where $\rho$ is the involution of $\Gmn$ defined in \eqref{eq-rho}.
  Using that the assignment $X_i\mapsto \epsilon_iX_i$ defines an
  automorphism of $\Gmn$, one easily checks that $\tilde\sigma$ is
  indeed an involution of $\Gmn$.  Moreover,
  \begin{align*}
    \phi(p(\epsilon_1X_1,\ldots,\epsilon_mX_m)) 
        &= p(\epsilon_1r_1,\ldots,\epsilon_mr_m)\\
        & = p(\sigma(r_1),\ldots,\sigma(r_m)) = \sigma(\phi(\rho(p)))\,.
  \end{align*}
  Replacing $p$ by $\rho(p)$, one sees that
  \begin{equation}\label{eqn3}
    \phi\circ\tilde\sigma=\sigma\circ\phi\,.
  \end{equation}
  In particular, the kernel $I$ of $\phi$ is $\tilde\sigma$-stable.
  Thus the involution $\tilde\sigma$ of $\Gmn$ induces
  involutions $\overline{\tilde\sigma}$ of $\Gmn/I$ and of
  $\Q(\Gmn/I)=k_n(X)$.  Since $\phi\circ\tilde\sigma=\sigma\circ\phi$,
  the $k$-algebra isomorphism $\bar\phi\colon k_n(X)\lra A$ is
  actually an isomorphism of central simple algebras with involution.
  In order to prove the proposition, we may thus assume that
  $A=k_n(X)$, that $R=\Gmn/I$, and that the image of the generic
  matrix $X_i$ in $R$ is $r_i$.  So $\phi$ is now the natural
  projection from $\Gmn$ onto $R=\Gmn/I$, and
  $\overline{\tilde\sigma}=\sigma$.

  We will now make $X=\mathcal{Z}(I)$ into a
  $\pntau$-variety, and show that the
  involution $\sigma$ on $A=k_n(X)$ is equal to $\sigma_{X,\tau}$.

  We first define an action of $\tau$ on $\Mnm$ as follows:
  \[\tau(a_1,\ldots,a_m)=(\epsilon_1a_1\tp,\ldots,\epsilon_ma_m\tp)\]
  Then for $p\in\Gmn$ and $a=(a_1,\ldots,a_m)\in\Mnm$,
  \begin{align*}
    p(\tau a) &= p(\epsilon_1a_1\tp,\ldots,\epsilon_ma_m\tp)\\
      & = \bigl([\rho(p(\epsilon_1X_1,\ldots,\epsilon_mX_m))]
             (a_1,\ldots,a_m)\bigr)\tp\\
      & = \bigl(\tilde\sigma(p)(a)\bigr)\tp\,.
  \end{align*}
  This shows that
  \begin{equation}\label{eqn4}
    \tilde\sigma(p)(a)=p(\tau a)\tp\,.
  \end{equation}
  It also shows that if $p\in I$ and $a\in X=\mathcal{Z}(I)$, then
  $p(\tau a)=0$ (since $\tilde\sigma(p)\in I$).  Hence $\tau(a)\in X$.
  Consequently, $\tau(X)=X$.
  
  We verify next that $\Mnm$ is a
  $\pntau$-variety.  To show this, it
  suffices to check that for $h\in\PGLn$ and
  $a=(a_1,\ldots,a_m)\in\Mnm$,
  \[[\tau h](a)=[(h\inv)\tp\tau](a)\,.\]
  Well,
  \begin{align*}
    [\tau h](a) &= \tau(ha_1h\inv,\ldots,ha_mh\inv)
        = \bigl(\epsilon_1(ha_1h\inv)\tp,\ldots,\epsilon_m(ha_mh\inv)\tp\bigr)\\
       &=\bigl((h\inv)\tp(\epsilon_1a_1\tp)h\tp,\ldots,
           (h\inv)\tp(\epsilon_ma_m\tp)h\tp\bigr)=[(h\inv)\tp\tau](a)\,.
  \end{align*}
  So $\Mnm$ is indeed a $\pntau$-variety,
  and so is its closed subvariety $X$ (which is both $\PGLn$- and
  $\langle\tau\rangle$-stable).  In particular, the action of $\tau$
  on $X$ defines now the involution $\sigma_{X,\tau}$ on $A=k_n(X)$.
  It only remains to show that $\sigma=\sigma_{X,\tau}$, i.e., that
  for all $f\in A$ and all $a\in X$ in general position,
  \[\sigma(f)(a)=f(\tau a)\tp\,,\]
  cf.~Lemma~\ref{lem4}.  It suffices to verify this for $f\in
  R=\Gmn/I$.  So let $p\in\Gmn$ such that $f=\phi(p)$.  Then by
  definition, $f(a)=p(a)$ for all $a\in X=\mathcal{Z}(I)$.  Hence
  \begin{align*}
    \sigma(f)(a) &\stackrel{\phantom{\eqref{eqn4}}}{=} [\sigma(\phi(p))](a)
                  \stackrel{\eqref{eqn3}}{=}[\phi(\tilde\sigma(p))](a)
                  = \tilde\sigma(p)(a)\\
                 &\stackrel{\eqref{eqn4}}{=}p(\tau a)\tp 
                  = [\phi(p)(\tau a)]\tp
                  = f(\tau a)\tp\,.\qedhere
  \end{align*}
\end{proof}

\begin{lemma}\label{lem9}
  Let $X$ and $Y$ be irreducible
  $\pntau$-varieties which are
  $\PGLn$-generically free.  Denote by $\sigma_{X,\tau}$ and
  $\sigma_{Y,\tau}$ the involutions of $k_n(X)$ and $k_n(Y)$ induced
  by the actions of $\tau$ on $X$ and $Y$, respectively.
  \begin{itemize}
  \item[(a)]Let $\varphi\colon X\dasharrow Y$ be a
    $\pntau$-equivariant dominant rational
    map.    
    Then the induced $k$-algebra homomorphism $\varphi^*\colon
    k_n(Y)\lra k_n(X)$ is a homomorphism of algebras with involution.
  \item[(b)] Let $\alpha\colon k_n(Y)\lra k_n(X)$ be a homomorphism of
    algebras with involution.  Then the induced $\PGLn$-equivariant
    dominant rational map $\alpha_*\colon X\dasharrow Y$ is
    $\pntau$-equivariant.
  \end{itemize}
\end{lemma}

\begin{proof}
  (a) Let $f\in k_n(Y)$.  We have to show that
  $\varphi^*(\sigma_{Y,\tau}(f))=\sigma_{X,\tau}(\varphi^*(f))$.  By
  definition, $\varphi^*(f)=f\circ\varphi$.  So we have to show that
  \[\sigma_{Y,\tau}(f)\circ\varphi=\sigma_{X,\tau}(f\circ\varphi)\,. \]
  By Lemma~\ref{lem4}, this reduces to showing that
  \[(f\circ\tau_Y)\circ\varphi=(f\circ\varphi)\circ\tau_X\,, \]
  where $\tau_X$ and $\tau_Y$ denote the actions of $\tau$ on $X$ and
  $Y$, respectively.  But this last equation is true since we are
  assuming that $\varphi$ is $\langle\tau\rangle$-equivariant.
  
  (b) We know that $\varphi=\alpha_*$ is $\PGLn$-equivariant.  So it
  suffices to show that $\tau_Y\circ\varphi=\varphi\circ\tau_X$.  Note
  that $\varphi^*=\alpha$.  Let $f\in k_n(Y)$.  We are assuming that
  $\varphi^*(\sigma_{Y,\tau}(f))=\sigma_{X,\tau}(\varphi^*(f))$.
  Arguing as in part~(a), we deduce from this that
  \[f\circ(\tau_Y\circ\varphi)=f\circ(\varphi\circ\tau_X)
    \qquad\text{for all $f\in k_n(Y)$.} \] 
  By the next lemma, this implies that
  $\tau_Y\circ\varphi=\varphi\circ\tau_X$. 
\end{proof}

\begin{lemma}
  Let $Y$ be an irreducible generically free $\PGLn$-variety, and $X$
  any variety.  Let $\varphi_1$ and $\varphi_2$ be two dominant
  rational maps from $X$ to $Y$ such that
  $f\circ\varphi_1=f\circ\varphi_2$ for all $f\in k_n(Y)$. Then
  $\varphi_1=\varphi_2$.
\end{lemma}

\begin{proof}
  By \cite[Lemma 8.1]{rv1}, $Y$ is birationally isomorphic (as
  $\PGLn$-variety) to a $\PGLn$-subvariety $Y'$ of $\Mnm$ for some
  integer $m\geq 2$.  Thus we may assume $Y=Y'$.  Substituting $f =
  p_i$ for $i = 1, \dots, m$, where $p_i \colon Y \lra \Mn$ is the
  projection from $Y \subseteq \Mnm$ to the $i$th copy of $\Mn$, we
  see that $\varphi_1=\varphi_2$.
\end{proof}

\begin{corollary}\label{cor-to-lem9}
  The associated $\pntau$-variety of a central simple algebra with
  involution is unique up to birational isomorphism of
  $\pntau$-varieties. 
\end{corollary}

\begin{proof}
  Let $(A,\sigma)$ and $X$ be as in Proposition~\ref{prop6}.  Say $Y$
  is another associated $\pntau$-variety of $(A,\sigma)$.  Then there
  is an isomorphism of algebras with involution $\alpha\colon
  k_n(Y)\lra k_n(X)$.  By \cite[Corollary~8.6]{rv1}, the induced
  $\PGLn$-equivariant dominant rational map $\alpha_*\colon
  X\dasharrow Y$ is a birational isomorphism, and it is
  $\pntau$-equivariant by Lemma~\ref{lem9}.
\end{proof}

The following proof is modeled after the proof of
\cite[Theorem~1.2]{rv1}. 

\begin{proof}[Proof of Theorem~\ref{thm}]
  Denote by $\mathcal C$ the category of central simple algebras of
  degree~$n$ with involution, and by $\mathcal P$ the category of
  irreducible $\pntau$-varieties which are generically free as
  $\PGLn$-varieties.  By Lemmas~\ref{lem4} and~\ref{lem9}(a), the
  assignment~$\mathcal F$ in Theorem~\ref{thm} is well-defined.  By
  \cite[Corollary~8.6(a)]{rv1}, $\mathcal F$ is a contravariant
  functor from $\mathcal P$ to $\mathcal C$.  If $\varphi$ and
  $\alpha$ are as in Lemma~\ref{lem9}, then $(\varphi^*)_*=\varphi$
  and $(\alpha_*)^*=\alpha$, see \cite[Proposition~8.4]{rv1}.  Hence
  $\mathcal F$ is full and faithful.  Moreover, by
  Proposition~\ref{prop6}, every object in $\mathcal C$ is isomorphic
  to the image of an object in $\mathcal P$.  Hence $\mathcal F$ is a
  covariant equivalence of categories between $\mathcal P$ and the
  dual category of $\mathcal C$, see, e.g., \cite[Theorem~7.6]{blyth}.
  In other words, $\mathcal F$ is a contravariant equivalence of
  categories between $\mathcal P$ and $\mathcal C$, as claimed.
\end{proof}

\section{Involutions of the First and Second Kind}
\label{sec:inv-kind}

In this section, we prove a characterization of involutions of the
first and second kind in terms of the generic freeness of the action
of $\pntau$ on the associated variety.  As an example, we then
consider the involution $\rho$ of $\UDmn$, which is of the first kind
if and only if $m=n=2$.

\begin{lemma}\label{lem:inv-second-kind}
  Let $X$ and $\sigma_{X,\tau}$ be as in Lemma~\ref{lem4}.  If
  $\sigma_{X,\tau}$ is an involution of the second kind, then the
  $\pntau$-action on $X$ is generically free.
\end{lemma}

Note that the converse is also true, see 
Corollary~\ref{cor:kind} below.

\begin{proof}
  Let $\sigma=\sigma_{X,\tau}$.  Recall that the center of $k_n(X)$ is
  $k(X)^\PGLn=k(Y)$, where $Y$ is any model of the rational quotient
  $X/\PGLn$ for the action of $\PGLn$ on $X$.  Since $\PGLn$ is a
  normal subgroup of $\pntau$, the $\pntau$-action on $k(X)$ induces
  an action of $\langle\tau\rangle\cong\pntau/\PGLn$ on $k(Y)$ as
  follows: For $f\in k(Y)$,
  $\tau(f)=f\circ\tau\inv=f\circ\tau=\sigma(f)$, where the last
  equality holds since the image of $f$ consists of scalar matrices.
  Thus $k(Y)^{\langle\tau\rangle}$ is the fixed field of the center of
  $k_n(X)$ under the action of the involution $\sigma$.  Since
  $\sigma$ is an involution of the second kind, $k(Y)$ is thus a field
  extension of $k(Y)^{\langle\tau\rangle}$ of degree~$2$.

  We can chose $Y$ to be a model for the rational quotient $X/\PGLn$
  such that the action of $\pntau$ on $X$ induces a regular action of
  $\langle\tau\rangle\cong\pntau/\PGLn$ on $Y$,
  cf.~\cite[Proposition~2.6 and Corollary to Theorem~1.1]{pv}.  Denote
  by $\pi_1\colon X\dasharrow Y$ and $\pi_2\colon Y\dasharrow
  Y/\langle\tau\rangle$ the rational quotient maps for the actions of
  $\PGLn$ and $\langle\tau\rangle$, respectively.  Then
  $Y/\langle\tau\rangle$ is a rational quotient for the
  $\pntau$-action on $X$, with quotient map $\pi_2\circ\pi_1$.

  Because $k(Y)$ is a field extension of $k(Y)^{\langle\tau\rangle}$
  degree~$2$, the fibers in general position of $\pi_2$ consist of two
  points.  The fibers in general position of $\pi_1$ consist of the
  distinct $\PGLn$-orbits in $X$.  Thus the fibers in general position
  of the quotient map $\pi_2\circ\pi_1$ for the rational quotient for
  the $\pntau$ action on $X$ consist of two distinct $\PGLn$-orbits.
  Since the $\PGLn$-action on $X$ is generically free, and since
  $\PGLn$ has index~$2$ in $\pntau$, it follows that the
  $\pntau$-action on $X$ is generically free.
\end{proof}

\begin{corollary}\label{cor:inv-first-kind}%
  Consider the involution $\sigma_{\tau,X}$ in Lemma~\ref{lem4}.  The
  following are equivalent:
  \begin{enumerate}
  \item $\sigma_{\tau,X}$ is an involution of the first kind.
  \item The $\pntau$-orbits in $X$ in general position are $\PGLn$-orbits.
  \item The $\pntau$-orbits in $X$ in general position are irreducible.
  \item The $\pntau$-action on $X$ is not generically free.
  \end{enumerate}
\end{corollary}

\begin{proof}
  Let $\sigma=\sigma_{\tau,X}$.  By definition, $\sigma$ is an
  involution of the first kind iff $\sigma(f)=f$ for all $f \in
  \Z(A)=k(X)^\PGLn$ .

  (1) $\Longrightarrow$ (2): For $f \in k(X)^\PGLn$, and for $x \in X$
  in general position,
  \[f(x) = \sigma(f)(x) = f(\tau x)\tp = f(\tau x)\,,\] 
  where the last equality follows from $f \in k(X)^\PGLn$.  Since
  finitely many $f \in k(X)^\PGLn$ separate the $\PGLn$-orbits in $X$
  in general position, this implies that $\tau(x)$ belongs to the
  $\PGLn$-orbit of $x$.  Hence the $\pntau$-orbits in general position
  in $X$ are just the $\PGLn$-orbits.

  (2) $\Longrightarrow$ (3):  Clear since $\PGLn$ is connected.

  (3) $\Longrightarrow$ (4):  Clear since $\pntau$ is not connected.

  (4) $\Longrightarrow$ (1): By Lemma~\ref{lem:inv-second-kind},
  $\sigma$ cannot be of the second kind.
\end{proof}

\begin{corollary}\label{cor:kind}
  Let $A$ be a central simple algebra of degree~$n$ with
  involution~$\sigma$.  Let $X$ be the associated $\pntau$-variety.
  Then $\sigma$ is an involution of the first kind if and only if the
  action of $\pntau$ on $X$ is not generically free.  Equivalently,
  $\sigma$ is an involution of the second kind if and only if the
  action of $\pntau$ on $X$ is generically free.
\end{corollary}

\begin{proof}
  By Proposition~\ref{prop6}, we may assume that
  $(A,\sigma)=(k_n(X),\sigma_{\tau,X})$, where $\sigma_{\tau,X}$ is as
  in Lemma~\ref{lem4}.  Now the result immediately follows from
  Corollary~\ref{cor:inv-first-kind}.
\end{proof}

\begin{remark}\label{rem:richardson}
  One can use a theorem of Richardson's to somewhat simplify the
  proofs in this section (though not in
  Section~\ref{sec:orth-and-symp}).  Let $X$ be an irreducible
  $\pntau$-variety which is $\PGLn$-generically free.  Since $\PGLn$
  has index~$2$ in $\pntau$, it follows that for $x\in X$ in general
  position, the stabilizer of $x$ in $\pntau$ has at most $2$
  elements.  In particular, these stabilizers are reductive.  Hence by
  a theorem of Richardson \cite[Theorem~9.3.1]{richardson}, there is a
  stabilizer in general position for the $\pntau$-action on $X$.
  Using this fact, one can now obtain Lemma~\ref{lem:inv-second-kind}
  as an immediate consequence of Corollary~\ref{cor:inv-first-kind},
  if one modifies the proof of the latter result as follows:

  (2) $\Longrightarrow$ (1): Let $f \in k(X)^\PGLn$. If $x \in X$ is
  in general position, there is an $h \in \PGLn$ such that
  $\tau(x)=h(x)$.  Then $\sigma(f)(x) = f(\tau x)\tp = f(\tau x) =
  f(hx) = f(x)$.  Hence $\sigma(f)=f$.

  (4) $\Longrightarrow$ (2): If a stabilizer in general position
  exists, it cannot be the trivial subgroup.  Since the $\PGLn$-action
  on $X$ is generically free, for $x\in X$ in general position, there
  is thus a $g\in\pntau\setminus\PGLn$ such that $gx=x$.  Since
  $\pntau=\PGLn\cup\PGLn g$, the $\pntau$-orbit of $x$ is equal to the
  $\PGLn$-orbit of $x$.
\end{remark}

\begin{example} \label{example:rho:first-kind} {\it The involution
    $\rho$ of $\UDmn$ {\upshape(}see
    Example~\ref{example:rho}{\upshape)} is of the first kind if and
    only if $m=n=2$.}  For $n\geq 3$, this follows
  from~\cite[Proposition~7.2]{reichstein:matrix-inv}, which states
  that the restriction of $\rho$ to the center of $\UDmn$ cannot be
  induced by an automorphism of the corresponding trace ring, so in
  particular cannot be the identity map.  Also the case $n=2$ can be
  handled with the methods of~\cite{reichstein:matrix-inv}.  For the
  convenience of the reader, we include the short arguments.

  Note that the (reduced) trace on $\UDmn$ is $\rho$-equi\-variant,
  see \cite[Corollaries~2.2 and~2.16]{KMRT}.  Assume first that
  $m=n=2$.  Denote the two generic matrices generating $\UD(2,2)$ by
  $X$ and $Y$.  Then the center of $\UD(2,2)$ is generated by
  $\tr(X)$, $\tr(Y)$, $\tr(X^2)$, $\tr(Y^2)$ and $\tr(XY)$.  Since
  $\rho$ fixes $X$, $Y$, $X^2$ and $Y^2$, it fixes the traces of these
  elements.  And $\rho(\tr(XY))=\tr(\rho(XY))=\tr(YX)=\tr(XY)$,
  showing that $\rho$ acts trivially on the center of $UD(2,2)$, so
  that $\rho$ is an involution of the first kind.  We will see in
  Section~\ref{sec:ex-rho-3} that $\rho$ is in this case an orthogonal
  involution.
  
  Assume next that $m\geq 3$, and let $X$, $Y$ and $Z$ be three
  distinct generic matrices in $\UDmn$.  Evaluating these generic
  matrices at the elementary matrix units $e_{1,2}$, $e_{2,2}$ and
  $e_{2,1}$, respectively, shows that $\tr(XYZ)\neq\tr(ZYX)$.  Since
  $\tr(ZYX)=\tr(\rho(XYZ))=\rho(\tr(XYZ))$, it follows that $\rho$ is
  of the second kind.
  
  Finally, assume that $m=2$ and $n\geq 3$.  Denote the two generic
  matrices generating $\UDmn$ by $X$ and $Y$.  Evaluating $X$ and $Y$
  at $e_{1,2}+e_{2,3}$ and $e_{1,2}+e_{3,1}$, respectively, shows that
  $\tr(XYX^2Y^2)\neq\tr(Y^2X^2YX)$.  Consequently, $\rho$ is of the
  second kind.
  \hfill\qedsymbol
\end{example}

\begin{remark}\label{rem:pntau-on-Mnm}
  Using Corollary~\ref{cor:inv-first-kind}, we can reinterpret this
  example as stating that that the action of $\pntau$ on $\Mnm$ is
  generically free if and only if $(m,n)\neq(2,2)$.
\end{remark}

\section{Preliminaries on Transposition in  $\PGLn$}
\label{Section4}

Given $a\in\GL_n$, we denote by $\bar a$ its image in $\PGLn$.  Let
$a,b\in\GL_n$ such that $\bar a=\bar b$, i.e., $b=\lambda a$ for some
nonzero scalar $\lambda$.  Then $b\tp=\lambda a\tp$.  Hence we can
define $\bar a\tp=\overline{a\tp}$.  One easily checks that for
$a,b\in\GL_n$, $(\bar a\tp)\tp=\bar a$, that $(\bar a \bar b)\tp =\bar
b\tp \bar a\tp$, and that $(\bar a\inv)\tp=(\bar a\tp)\inv$.

Now let $a\in\GL_n$ such that $\bar a=\bar a\tp$.  Then $a\tp=\epsilon
a$ for some nonzero scalar~$\epsilon$.  Denoting the $i,j$-entry of
$a$ by $a_{i,j}$, it follows that $a_{i,j}=\epsilon
a_{j,i}=\epsilon^2a_{i,j}$.  Consequently $\epsilon=\pm1$, i.e., $a$
is either a symmetric or a skew-symmetric matrix.  If $b\in\GL_n$ with
$\bar b=\bar a$, then $b\tp=\epsilon b$.  This allows us to make the
following definition.

\begin{definition}\label{defn:symm:PGLn}
  Let $g\in\PGLn$ with $g\tp=g$.  We call $g$ {\it symmetric}
  (respectively, {\it skew-symmetric}) if one (and hence all)
  preimages of $g$ in $\GL_n$ are symmetric (respectively,
  skew-symmetric).
\end{definition}

\begin{remark}\label{sec4:rem1}
  Let $a\in \GL_n$ such that $a\tp=\epsilon a$ with $\epsilon=\pm1$.
  Then $\det(a)=\det(a\tp)=\epsilon^n\det(a)$, implying
  $\epsilon^n=1$.  Since $\epsilon=\pm1$, this implies $\epsilon=1$ if
  $n$ is odd.  So for $n$ odd, there are no skew-symmetric elements in
  $\PGLn$.  As we will see, this corresponds to the fact that
  involutions of the first kind on central simple algebras of odd
  degree cannot by symplectic, see \cite[Corollary~2.8(1)]{KMRT}.
\end{remark}

\begin{remark}\label{sec4:rem2}
  For future use, we record a few technical facts.  Again, let $a\in
  \GL_n$ such that $a\tp=\epsilon a$ with $\epsilon=\pm1$.  Then
  $(a\inv)\tp=\epsilon\inv a\inv=\epsilon a\inv$.  Hence if
  $g\in\PGLn$ is symmetric (respectively, skew-symmetric) then so is
  $g\inv$.  Recall that a matrix is {\em congruent} to $a$ if it is of
  the form $bab\tp$ for some $b\in \GL_n$.  Note that
  $(bab\tp)\tp=ba\tp b\tp=\epsilon(bab\tp)$.  Consequently, if
  $g\in\PGLn$ is symmetric (respectively, skew-symmetric), then so are
  $hgh\tp$ and $hg\inv h\tp$ for any $h\in \PGLn$.
\end{remark}

\begin{remark}\label{sec4:rem3-new}
  Let $g\in\PGLn$ be symmetric or skew-symmetric.  Given any
  $h\in\PGLn$, $\tau g$ is conjugate in $\pntau$ to $h\inv(\tau g)
  h=\tau(h\tp gh)$.  Recall that any non-singular symmetric matrix
  over an algebraically closed field of characteristic $\neq2$ is
  congruent to the identity matrix, while any non-singular
  skew-symmetric matrix is congruent to
  \begin{equation}\label{sec4:eq1-new}
    J=\begin{pmatrix} 0 & I\\ -I & 0 \end{pmatrix}\,,
  \end{equation}
  where $I$ is the identity matrix of order $\frac{n}{2}$ (see, e.g.,
  \cite[Sections~8.3 and 8.6]{cohn-1}).
  Consequently, if $g$ is symmetric, $\tau g$ is conjugate in $\pntau$
  to $\tau$, while if $g$ is skew-symmetric, $\tau g$ is conjugate in
  $\pntau$ to $\tau g_0$, where $g_0$ is the image in $\PGLn$ of the
  matrix~$J$.
\end{remark}

\section{Proof of Theorem~\ref{thm-stab}}
\label{sec:orth-and-symp}

By Proposition~\ref{prop6}, we may assume that
$(A,\sigma)=(k_n(X),\sigma_{\tau,X})$, where $\sigma_{\tau,X}$ is as
in Lemma~\ref{lem4}.  Part~(b) of Theorem~\ref{thm-stab} was already
proved in Corollary~\ref{cor:kind}.  Note that part~(a) will
immediately follow from part~(b), once we proved the existence of the
stabilizer in general position.

If $\sigma$ is an involution of the second kind, we already saw that
the action of $\pntau$ on $X$ is generically free and thus has a
stabilizer in general position, namely $S=\{1\}$.

Assume now that $\sigma$ is an involution of the first kind.  Since
the $\PGLn$-action on $X$ is generically free, and since $\PGLn$ has
index~$2$ in $\pntau$, Corollary~\ref{cor:inv-first-kind} implies that
every $x\in X$ in general position must have a stabilizer in $\pntau$
of order~$2$.  The nontrivial element in the stabilizer must be of
form $\tau g_x$ for some $g_x\in\PGLn$ (depending on $x$).  Note that
then $\tau x=g_xx$.  Hence $x=\tau^2x=\tau g_xx=((g_x)\inv)\tp \tau
x=((g_x)\inv)\tp g_x x$.  Since $((g_x)\inv)\tp g_x=((g_x)\tp)\inv
g_x$ belongs to $\PGLn$, it must be $1$, i.e., $(g_x)\tp=g_x$.  Hence
$g_x$ is either a symmetric or a skew-symmetric element of $\PGLn$
(see Definition~\ref{defn:symm:PGLn}).

Let $F$ be the center of $A$.  We denote by $\Sym(A,\sigma)$ and
$\Skew(A,\sigma)$ the $F$-subspaces of symmetric and skew-symmetric
elements of $A$ with respect to $\sigma$.  Let $\{a_i\}$ and $\{b_j\}$
be $F$-bases of $\Sym(A,\sigma)$ and $\Skew(A,\sigma)$, respectively.
For $x\in X$ in general position, $\{a_i(x)\}$ and $\{b_j(x)\}$ are
$k$-linearly independent subsets of $\Mn$ by
\cite[Lemma~7.4(a)]{reichstein}.
  
Denote by $U$ a dense open subset of $X$ such that all $x\in U$ have
the following properties: $a_i(x)$ and $b_j(x)$ are defined for all
$i$ and $j$, and the stabilizer of $x$ in $\pntau$ has two elements
(so that $g_x$ is defined).

\begin{lemma}\label{lem:U}
  Let $x\in U$.  Then $g_x$ is symmetric iff $\sigma$ is orthogonal.
\end{lemma}

\begin{proof}
  For simplicity, set $g=g_x$.  Let $u$ be a preimage of $g\inv$ in
  $\GL_n$.  So $u\tp=\pm u$.  Denote by $\Int(u)$ conjugation by $u$
  on $\Mn$.  Transposition defines an involution of orthogonal type of
  $\Mn$.  By \cite[Proposition 2.7(1)]{KMRT},
  \[\tilde\sigma=\Int(u)\circ\operatorname{transposition}\]
  is an involution of $\Mn$ of the first kind.  Moreover, by
  \cite[Proposition 2.7(3)]{KMRT}, $\tilde\sigma$ is of the same type
  as transposition (i.e., is of orthogonal type) iff $u\tp=u$.
  Consequently, $\tilde\sigma$ is of orthogonal type iff $u\tp=u$ iff
  $g\inv$ is symmetric iff $g$ is symmetric.

  Let $f$ be any element of $A$ defined at $x$.  Using the definition
  of $\sigma$ (see Lemma~\ref{lem4}) and the facts that $f$ is
  $\PGLn$-equivariant and that $u\tp=\pm u$, we see that
  \begin{align*}
    \sigma(f)(x)&=f(\tau x)\tp=f(gx)\tp=(g\cdot f(x))\tp=(u\inv
    f(x)u)\tp\\
    &=u\tp f(x)\tp (u\inv)\tp=uf(x)\tp u\inv=\tilde\sigma(f(x))\,.
  \end{align*}
  Hence $a_i(x)=\sigma(a_i)(x)=\tilde\sigma(a_i(x))$, and
  $-b_j(x)=\sigma(b_j)(x)=\tilde\sigma(b_j(x))$.  Thus
  $\dim_k(\Sym(\Mn,\tilde\sigma)) \geq \dim_F(\Sym(A,\sigma))$, and
  $\dim_k(\Skew(\Mn,\tilde\sigma)) \geq \dim_F(\Skew(A,\sigma))$.  It
  follows by \cite[Proposition 2.6(1)]{KMRT} that $\sigma$ and
  $\tilde\sigma$ are of the same type.  Hence $\sigma$ is of
  orthogonal type iff $\tilde\sigma$ is of orthogonal type iff $g$ is
  symmetric.
\end{proof}

Lemma~\ref{lem:U} has several immediate consequences.  First, for
$x\in U$, $g_x$ is skew-symmetric iff $\sigma$ is symplectic.
Secondly, either $g_x$ is symmetric for all $x\in U$, or $g_x$ is
skew-symmetric for all $x\in U$.

Now let $x,y\in U$.  So $g_x$ and $g_y$ are either both symmetric or
both skew-symmetric.  Hence there is an $h\in \PGLn$ such that
$g_y=h\tp g_xh$, cf.~Remark~\ref{sec4:rem3-new}.  Consequently, $\tau
g_y=\tau(h\tp g_xh)=h\inv(\tau g_x)h$ is conjugate to $\tau g_x$ in
$\pntau$.  Since the stabilizers of $x$ and $y$ in $\pntau$ are
$\{1,\tau g_x\}$ and $\{1,\tau g_y\}$, respectively, they are
conjugate in $\pntau$.  Since this is true for all $x$ and $y$ in the
dense open subset $U$ of $X$, it follows that there is a stabilizer in
general position for the $\pntau$-action on $X$.  Moreover, by
Remark~\ref{sec4:rem3-new}, the stabilizer in general position will be
conjugate to $\{1,\tau\}$ if $\sigma$ is orthogonal, and to $\{1,\tau
g_0\}$ if $\sigma$ is symplectic.  This completes the proof of
Theorem~\ref{thm-stab}.

\section{The Involution $\rho$ of $\UD(2,2)$}
\label{sec:ex-rho-3}

We now illustrate Theorem~\ref{thm-stab} %
in the case of the universal division algebra $\UD(2,2)$.  Recall from
Example~\ref{example:rho} that the $\pntau$-variety associated to
$(\UDmn,\rho)$ is $\Mnm$, where $\PGLn$ acts by component-wise
conjugation and $\tau$ acts by component-wise transposition.  We
observed in Example~\ref{example:rho:first-kind} that the involution
$\rho$ of $A=\UD(2,2)$ is an involution of the first kind, and that it
fixes both of the generic matrices generating~$A$.  Hence
$\dim\Sym(A,\rho)\geq 2$.  Since $\deg A=n=2$,
\cite[Proposition~2.6(1)]{KMRT} implies that $\rho$ is of orthogonal
type.  Hence by Theorem~\ref{thm-stab}(c), the action of $\tau$ on
$(\M_2)^2$ is of {\em symmetric type} (using the terminology of
Remark~\ref{rem:1.3}).  We now verify this directly.

Let $x=(x_1,x_2)\in(\M_2)^2$ be in general position.  We may assume
that $x_1$ is diagonalizable, and that $\M_2=k\{x_1,x_2\}$.  Hence
there is an $h\in\PGL_2$ such that $y=h\cdot x$ is of form
\[y=(y_1,y_2)=\bigl(\begin{pmatrix}\lambda&0\\0&\mu\end{pmatrix},
\begin{pmatrix}a&b\\c&d\end{pmatrix}\bigr)\,,\]
with $\lambda,\mu,a,b,c,d\in k$.  Since
$\M_2=k\{x_1,x_2\}=k\{y_1,y_2\}$, both $b,c\neq0$.  Denote by $g$ the
image in $\PGL_2$ of the diagonal matrix
$\bigl(\begin{smallmatrix}c&0\\0&b\end{smallmatrix}\bigr)$.
Note that $g$ is symmetric according to
Definition~\ref{defn:symm:PGLn}.  A short calculation shows that
$g\cdot y=y\tp$.  Hence $(gh)\cdot x=g\cdot y=y\tp=(h\cdot
x)\tp=(h\inv)\tp \cdot x\tp$, so that $\tau\cdot x=x\tp=(h\tp gh)\cdot
x$.  As we saw in Remark~\ref{sec4:rem2}, $h\tp gh$ is symmetric since
$g$ is.  It follows that the action of $\tau$ on $(\M_2)^2$ is of
symmetric type.

\section{Models for Associated Varieties}
\label{sec:new}

Up to birational isomorphism, the associated variety of a central
simple algebra can always be found in $\Mnm$ for some $m\geq2$, where
$\PGLn$ acts on $\Mnm$ by simultaneous conjugation.  To be more
precise, denote by $U_{m,n}$ the subset of $\Mnm$ consisting of all
tuples $(a_1,\ldots,a_m)$ in $\Mnm$ such that $a_1$, \ldots, $a_m$
generate $\Mn$ as $k$-algebra.  An $n$-variety is a closed
$\PGLn$-invariant subvariety of $U_{m,n}$ for some $m\geq2$.  It is
shown in~\cite[Lemma~8.1]{rv1} that every irreducible generically free
$\PGLn$-variety is $\PGLn$-equivariantly birationally isomorphic to an
irreducible $n$-variety in $U_{m,n}$ for some $m$.  Adapting the proof
of this lemma, we now obtain the corresponding result for the
associated varieties of central simple algebras with involution.

\begin{lemma}\label{lem:new}
  Let $X$ be an irreducible $\pntau$-variety which is generically free
  as $\PGLn$-variety.  If the action of $\pntau$ is of skew-symmetric
  type, assume additionally that $n>2$.  There is an irreducible
  $n$-variety $Y$ in $\Mnm$ for some $m\geq 2$ which is invariant
  under component-wise transposition {\upshape(}and thus a
  $\pntau$-variety{\upshape)} such that $X$ is $\pntau$-equivariantly
  isomorphic to $Y$.
\end{lemma}

Here $\pntau$ acts on $\Mnm$ and $Y$ as in Example~\ref{example:rho}.
If the action of $\pntau$ is of skew-symmetric type (see
Remark~\ref{rem:1.3}), the result is not true in case $n=2$: an easy
calculation shows that $U_{2,m}$ does not contain points whose
stabilizer in $\pntau$ is $\langle \tau
\bigl[\begin{smallmatrix}0&\,\,1\\-1&\,\,0\end{smallmatrix}\bigr]\rangle$.

\begin{proof}
  The $\pntau$-action on $X$ has a stabilizer $H$ in general position
  as in Theorem~\ref{thm-stab}.  Assume first that $H=1$.  Then the
  proof of \cite[Lemma~8.1]{rv1} goes through with minor changes, if
  one replaces $\PGLn$ by $\pntau$ throughout: By
  Remark~\ref{rem:pntau-on-Mnm}, there is a dense open subset $V$ of
  $(\Mn)^3$ such that every element of $V$ has trivial stabilizer in
  $\pntau$.  Now set $m=r+3$, and construct a $\pntau$-equivariant
  rational map $\phi\colon X\dasharrow (\Mn)^3$ whose image meets
  $U_{3,n}\cap V$; the latter enables us to choose the subset
  $S\subset X$ in such a way that all points of $f(S)$ have trivial
  stabilizer in $\pntau$.

  Now assume that $H=\langle\tau\rangle$ or that $H=\langle\tau
  g_0\rangle$ where $g_0 =
  \bigl[\begin{smallmatrix}0&\,\,I\\-I&\,\,0\end{smallmatrix}\bigr]$.
  Since both $\pntau$ and $H$ are reductive, $X$ is
  $\pntau$-equivariantly birationally isomorphic to a stable affine
  $\pntau$-variety, see \cite[Theorem~1.1]{rv-affine-stable}.  Hence
  we may assume that $X$ is affine, and that the $\pntau$-action on
  $X$ is stable.  Pick $x_0\in X$ with stabilizer $H$ and such that
  the $\pntau$-orbit of $x_0$ is closed.  For some $q$, we can find
  $a_0\in U_{q,n}\subset(\Mn)^q$ such that also $a_0$ has stabilizer
  $H$: If $H=\langle\tau\rangle$, this is clear since $\Mn$ can be
  generated as $k$-algebra by $q$ symmetric matrices for some~$q$.
  And an easy computation deals with the case $n\geq4$ and
  $H=\langle\tau g_0\rangle$.  It follows by \cite[Lemma~2.6]{rv2} (a
  slight extension of \cite[Theorem 1.7.12]{popov}) that there is a
  $\pntau$-equivariant morphism $\phi\colon X\to(\Mn)^q$ with
  $\phi(x_0)=a_0$.

  It follows from the proof of Lemma~\ref{lem:inv-second-kind} (or
  from Corollary~\ref{cor:inv-first-kind}(b) and
  \cite[Remark~2.5]{reichstein}), that under the current hypotheses,
  $X/\pntau=X/\PGLn$, so that $k(X)^{\pntau}=k(X)^\PGLn$.

  Using the map $\phi$ we just constructed and setting $m=r+q$, the
  remainder of the proof of \cite[Lemma~8.1]{rv1} goes through nearly
  literally.  Note that since $k(X)^\PGLn=k(X)^{\pntau}$, $f$ is
  $\pntau$-equivariant.
\end{proof}

\end{document}